\documentclass{article}

\usepackage{latexsym}
\usepackage{amscd}
\usepackage{graphics}
\usepackage{amsmath}
\usepackage{amssymb}
\usepackage{mathrsfs}
\usepackage{amsthm}
\usepackage{bbm}
\input xy
\xyoption{all}

\newcommand{\N}{\mathbb{N}}                     
\newcommand{\Z}{\mathbb{Z}}                     
\newcommand{\R}{\mathbb{R}}                     
\newcommand{\C}{\mathbb{C}}                     
\newcommand{\T}{\mathbb{T}}                     
\newcommand{\set}[2]{\left\{{#1}\mid{#2}\right\}}       
\newcommand{\im}{\mathrm{Im\,}}                 
\newcommand{\re}{\mathrm{Re\,}}                 
\newcommand{\coker}{\mathrm{coker\,}}           

\newtheorem{thm}{\bf Theorem}[section]      
\newtheorem*{thm*}{\bf Theorem}			
\newtheorem{lem}[thm]{\bf Lemma}            
\newtheorem{prop}[thm]{\bf Proposition}     

\title{Corrigendum: ``On the Floer homology of cotangent bundles''}

\author{Alberto Abbondandolo and Matthias Schwarz}

\date{}

\begin{document}    

\maketitle   

\begin{abstract}
We fix an orientation issue which appears in our previous paper about the isomorphism between Floer homology of cotangent bundles and loop space homology. When the second Stiefel-Whitney class of the underlying manifold does not vanish on 2-tori, this isomorphism requires the use of a twisted version of the Floer complex.
\end{abstract}

\section*{Introduction}

The aim of this ``Corrigendum'' is to fix an issue with orientations which appears in our paper \cite{as06}. In this paper, we consider a fiber-wise time-periodic uniformly convex and asymptotically quadratic Lagrangian function $L\in C^{\infty}( \T \times TM)$ on the tangent bundle of a closed oriented manifold $M$ and its Legendre-dual Hamiltonian $H\in C^{\infty}( \T \times T^*M)$. Then we construct an isomorphism from the Morse complex of the Lagrangian action functional, which is given by $L$, to the Floer complex, which is associated to $H$ and to an almost complex structure which is compatible with the standard symplectic form $\omega$ of $T^*M$. We deal with both periodic and Dirichlet boundary conditions. 
In the periodic case, the existence of this isomorphism implies that the homology of the Floer complex, or equivalently the symplectic homology of the unit cotangent disc bundle $D^* M$, is isomorphic to the singular homology of $\Lambda(M)$, the free loop space of $M$. This result had been previously proved by different methods by C.\ Viterbo \cite{vit03} and by D.\ Salamon and J.\ Weber \cite{sw06}. In the Dirichlet case, the corresponding Floer homology is isomorphic to the singular homology of the based loop space of $M$.

However, a recent work of T.\ Kragh \cite{kra07} has highlighted some sign discrepancies between symplectic homology of cotangent disc bundles, defined by standard conventions regarding signs, and loop space homology (see also \cite{kra11}). P.\ Seidel has confirmed the existence of these discrepancies in the informal note \cite{sei10}, where he shows that, again using standard conventions, the symplectic homology of the cotangent disc bundle of $\C\mathbb{P}^2$ over a field of characteristic different from two vanishes, so in particular it is not isomorphic to the singular homology of the loop space of $\C\mathbb{P}^2$ over the same field. M.\ Abouzaid has clarified this issue, by showing that when the second Stiefel-Whitney class of $M$ does not vanish on 2-tori, the isomorphism between 
symplectic homology of $D^*M$ and singular homology of the free loop space of $M$ requires either to use a non-trivial system of local coefficients for the singular homology of $\Lambda(M)$, or to change the definition of symplectic homology by a twist (see \cite[Section 3]{abo11}).

This means that the isomorphism $\Theta$ that we construct in our Theorem 3.1 in \cite{as06} fails to be a chain map in the periodic case, when the second Stiefel-Whitney class of $M$ does not vanish on 2-tori. After a careful inspection of our proofs, we can now clarify the source of the mistake in \cite{as06}: in order to orient the spaces of half-cylinders 
\[
u : [0,+\infty[\times \T \longrightarrow T^*M
\]
on which the definition of $\Theta$ is based, we use the fact that the orientation of the determinant line which is associated to the linearization of the Floer equation along $u$ does not depend on the choice of the vertical-preserving unitary trivialization of $u^*(TT^*M)$ coinciding with a given one at $\{+\infty\} \times \T$. Such a trivialization is used in order to produce a Fredholm operator between fixed Banach spaces. However, in order to prove that $\Theta$ is a chain map, one would need the same fact for a more general class of trivializations: the hypothesis that the trivialization should be vertical-preserving should be required only at $\{0\} \times \T$. This is due to the fact that if $v: \R \times \T \rightarrow T^*M$ is a Floer cylinder which joins two given periodic orbits, along which two vertical-preserving unitary trivializations of $TT^*M$ have been fixed, it is always possible to find a unitary trivialization of $v^*(TT^*M)$ which agrees with the given ones at $\{-\infty\}\times \T$ and $\{+\infty\}\times \T$, but in general we cannot require this trivialization to be vertical-preserving. The latter fact had been correctly noticed in \cite{as06}, but we had overlooked its consequence, namely that this requires to use the above more general class of trivialization also for half-cylinders. However,
the invariance of the orientation of the determinant line with respect to this more general change of trivialization is simply not true, as we show in Proposition \ref{complex} below. 

Using this proposition, we then explain how our Theorem 3.1 can be corrected by replacing the standard Floer complex by a twisted version of it, exactly as M.\ Abouzaid does in \cite{abo11}, but using the different approach for dealing with orientations that we use in our original paper \cite{as06} (which is essentially an adaptation of A.\ Floer and H.\ Hofer's approach to the particular case of cotangent bundles, see \cite{fh93}). 

The part of Theorem 3.1 which concerns Dirichlet boundary conditions needs no correction.   

\paragraph{Acknowledgments.} We wish to thank P.\ Seidel for sharing with us his unpublished computations \cite{sei10} on the symplectic homology of $D^* \C \mathbb{P}^2$. We are particularly grateful to M.\ Abouzaid for patiently discussing with us  the issue which we fix in this ``Corrigendum''. The present work is part of the authors' activities within CAST, a Research Network Program of the European Science Foundation.

\section{The effect of certain unitary conjugacies on a class of Cauchy-Riemann operators on half-cylinders}
\label{uno}

Let $\R^{2n}=\R^n \times \R^n$ be endowed with the standard symplectic structure
\[
\omega_0 = dp \wedge dq, \quad (q,p)\in \R^n \times \R^n,
\]
and with the standard complex structure
\[
J_0 = \left( \begin{array}{cc} 0 & I \\ - I & 0 \end{array} \right).
\]
The symbol $\lambda_0$ denotes the vertical Lagrangian subspace $(0) \times \R^n$, and 
\[
W^{1,p}_{\lambda_0} ( ]0,+\infty[ \times \T, \R^{2n})
\]
denotes the space of maps $u:  ]0,+\infty[ \times \T \rightarrow \R^{2n}$ of Sobolev class $W^{1,p}$ whose trace on the boundary $\{0\} \times \T$ is $\lambda_0$-valued.

Let $\mathfrak{gl}(2n)$ be the vector space of linear endomorphisms of $\R^{2n}$. The symbol $\Sigma^+$ denotes the space of linear operators 
\[
D_S^+ : W^{1,p}_{\lambda_0} ( ]0,+\infty[ \times \T, \R^{2n}) \rightarrow L^p ( ]0,+\infty[ \times \T, \R^{2n})
\]
of the form
\[
u \mapsto \partial_s u - J_0 \partial_t u - S(s,t) u,
\]
where 
\[
S\in C^0( [0,+\infty] \times \T, \mathfrak{gl}(2n))
\]
is such that $S(+\infty,t)$ is symmetric for every $t\in \T$, and the loop $S(+\infty,\cdot)$ is non-degenerate, meaning that $\gamma(1)$ does not have the eigenvalue 1, where $\gamma$ is
the path of symplectic matrices which solves the linear Cauchy problem
\[
\frac{d}{dt} \gamma(t) = J_0 S(+\infty,t) \gamma(t), \qquad \gamma(0)=I.
\]
If $2<p<\infty$, $D_S^+$ is a Fredholm operator of index $-\mu_{CZ}(\gamma)$, minus the Conley-Zehnder index of the symplectic path $\gamma$ (see \cite[Theorem 3.4]{as06}).

Let $\mathrm{Sym}(2n)$ be the vector space of symmetric linear endomorphisms of $\R^{2n}$.
If $S^+ \in C^0(\T, \mathrm{Sym}(2n))$ is a non-degenerate loop, the symbol $\Sigma^+(S^+)$ denotes the subspace of all $D_S\in \Sigma^+$ with given asymptotics  $S(+\infty,\cdot) = S^+$. Since $\Sigma^+(S^+)$ is contractible (being star-shaped), the restriction of the determinant bundle to it, that we denote by $\det(\Sigma^+(S^+))$, is trivial, hence orientable.

We recall that the determinant line of a Fredholm operator $A$ is the one-dimensio\-nal real vector space
\[
\det(A) := \Lambda^{\max}( \ker A) \otimes \Lambda^{\max} \bigl((\coker A)^*\bigr),
\]
where $\Lambda^{\max}(V)$ denotes the top exterior power of the finite dimensional real vector space $V$. The collection of the determinant lines of all Fredholm operators from a Banach space $E$ to a Banach space $F$ has the structure of a smooth real line bundle over the space $\mathrm{Fred}(E,F)$ of Fredholm operators, which is known as the determinant bundle and is denoted by $\det(E,F)$. Two isomorphisms $\Phi: E \cong E'$ and $\Psi: F \cong F'$ induce a canonical smooth line bundle isomorphism $\det(E,F) \rightarrow \det(E',F')$ lifting the diffeomorphism
\[
\mathrm{Fred}(E,F) \rightarrow \mathrm{Fred}(E',F'), \qquad A \mapsto \Psi A \Phi^{-1}.
\]

Let us identify $\R^{2n}$ with $\C^n$ by the mapping $(q,p)\mapsto q+ip$.
Let $\mathrm{U}(n)$ be the unitary group of $\C^n$. The subgroup of $\mathrm{U}(n)$ of automorphisms which preserve the vertical subspace $\lambda_0=i\R^n$ is precisely $\mathrm{O}(n)$, the orthogonal group of $\R^n$, whose elements are extended to $\C^n$ by complex linearity. If we also impose that the restriction to $\lambda_0$ should be orientation-preserving, we obtain the subgroup $\mathrm{SO}(n)$.

We wish to study the behavior of the determinant line of the operator $D_S^+$ by conjugation by an $(s,t)$-dependent unitary transformation. We recall that the extended real line $[-\infty,+\infty]$ is given the differentiable structure which is induced by the homeomorphism $[-\infty,+\infty] \cong [-\pi/2,\pi/2]$ which extends the function $s\rightarrow \arctan s$. If $U\in C^{\infty}([0,+\infty] \times \T, \mathrm{U}(n))$ then 
\begin{equation}
\label{ide}
U (\partial_s - J_0 \partial_t - S) U^{-1} = \partial_s - J_0 \partial_t  - (\partial_s U) U^{-1} + J_0 (\partial_t U) U^{-1} - U S U^{-1}.
\end{equation}
In particular, if $U(0,t)\in \mathrm{O}(n)$ and $U(+\infty,t)=I$ for every $t\in \T$, then the left multiplication by $U^{-1}$ is an automorphism of $W^{1,p}_{\lambda_0} ( ]0,+\infty[ \times \T, \R^{2n})$, the left multiplication by $U$ is an automorphism of $L^p ( ]0,+\infty[ \times \T, \R^{2n})$,
and the operators $D_S$ and $U D_S U^{-1}$ belong to the same space $\Sigma^+(S^+)$.

At the end of Section 3.2 in \cite{as06}, we deal with ``the analogue of Lemma 13 in \cite{fh93}''. In the context of that section, this refers to the following proposition:

\begin{prop}
\label{easy}
Let $S^+ \in C^0(\T, \mathrm{Sym}(2n))$ be a non-degenerate loop.
Let $U\in C^{\infty}([0,+\infty] \times \T, \mathrm{SO}(n))$ be such that $U(+\infty,t)=I$ for every $t\in \T$. Then the canonical lift to the determinant bundle of the map
\[
\Sigma^+(S^+) \rightarrow \Sigma^+(S^+), \qquad D_S^+ \mapsto U D_S^+ U^{-1},
\]
is orientation preserving.
\end{prop}

This Proposition is indeed correct and easy to prove, by a simple homotopy argument. However, in order to prove that the isomorphism $\Theta$ of Theorem 3.1 in \cite{as06} is a chain map, one would need an analogous statement for a more general class of conjugacies $U$: $U$ should be in 
$C^{\infty}([0,+\infty] \times \T, \mathrm{U}(n))$ and such that $U(+\infty,t)=I$, $U(0,t)\in \mathrm{SO}(n)$ for every $t\in \T$ (the reason for this is explained in Section \ref{tre} below). But the canonical lift to the determinant bundle of the map $D_S^+ \mapsto U D_S^+ U^{-1}$ need not be orientation-preserving when $U$ is in the latter more general class. Indeed, the correct generalization of Proposition \ref{easy} to the above class of conjugacies is:

\begin{prop}
\label{complex}
Let $S^+ \in C^0(\T, \mathrm{Sym}(2n))$ be a non-degenerate loop. Let $U\in C^{\infty}([0,+\infty] \times \T, \mathrm{U}(n))$ be such that $U(+\infty,t)=I$ and $U(0,t)\in \mathrm{SO}(n)$ for every $t\in \T$. Then the canonical lift to the determinant bundle of the map
\begin{equation}
\label{mappa}
\Sigma^+(S^+) \rightarrow \Sigma^+(S^+), \qquad D_S^+ \mapsto U D_S^+ U^{-1},
\end{equation}
is orientation preserving if and only if either
\begin{enumerate}
\item $n=1$, or
\item $n=2$ and the homotopy class of the loop $U(0,\cdot)$ in $\mathrm{SO}(2)$ is an even multiple of the generator of $\pi_1(\mathrm{SO}(2))=\Z$, or
\item $n\geq 3$ and the loop $U(0,\cdot)$ is contractible in $\mathrm{SO}(n)$.
\end{enumerate}
\end{prop}

Since $\mathrm{SO}(1)=\{1\}$, $\pi_1(\mathrm{SO}(2))=\Z$ and $\pi_1\mathrm({SO}(n))=\Z_2$ for every $n\geq 3$, conditions (i), (ii) and (iii) can be restated in a unified way by saying that the loop $U(0,\cdot)$ in $\mathrm{SO}(n)$ lifts to a loop in the two-fold covering $\mathrm{Spin}(n)$ of $\mathrm{SO}(n)$.
The proof of this proposition is rather long and is presented in Section \ref{propro} below. It is based on homotopy arguments, together with some explicit computations which are similar to those appearing in \cite{des98} and in the proof of \cite[Proposition 8.1.7]{fooo09}.

\section{The twisted Floer complex}

In this section we recall how signs are determined in the definition of the standard Floer complex for periodic Hamiltonian orbits on the cotangent bundle of a closed manifold, and we explain how this definition can be modified in order to obtain a twisted Floer complex. The latter definition agrees with the one given by M.\ Abouzaid in \cite[Secton 3]{abo11}, which however is presented using a different and more intrinsic approach. Here we prefer to rephrase everything using the approach and the notation of \cite{as06}.
 
If $S^-,S^+\in C^0(\T,\mathrm{Sym}(2n))$ are two non-degenerate loops, $\Sigma(S^-,S^+)$ denotes the space of all operators of the form
\[
D_S : W^{1,p}(\R\times \T,\R^{2n}) \rightarrow L^p(\R\times \T,\R^{2n}), \qquad D_S u := \partial_s u - J_0 \partial_t u - S(s,t) u,
\]
where $S\in C^0(\overline{\R} \times \T, \mathfrak{gl}(2n))$ has asymptotics
\[
S(-\infty,\cdot) = S^-, \qquad S(+\infty,\cdot) = S^+.
\]
The space $\Sigma(S^-,S^+)$ consists of Fredholm operators and the restriction of the determinant bundle to it, that we denote by $\det (\Sigma(S^-,S^+))$, is trivial, hence orientable. 

Two orientations $o(S_1,S_2)$ and $o(S_2,S_3)$ of $ \det (\Sigma(S_1,S_2))$ and $\det (\Sigma(S_2,S_3))$ can be glued and induce an orientation
\[
o(S_1,S_2) \# o(S_2,S_3)
\]
of $\det (\Sigma(S_1,S_3))$. Such a gluing construction is associative. A coherent orientation for $\Sigma$ is the choice of an orientation $o(S^-,S^+)$ of $\det (\Sigma(S^-,S^+))$ for each pair $(S^-,S^+)$ of non-degenerate loops, such that
\[
o(S_1,S_3) = o(S_1,S_2) \# o(S_2,S_3),
\]
for each triplet $(S_1,S_2,S_3)$. The existence of (uncountably many) coherent orientations for $\Sigma$ is proved in \cite[Theorem 12]{fh93}.

Let $M$ be a closed oriented manifold of dimension $n$ and let $H\in C^{\infty}(\T\times T^*M)$ be a Hamiltonian. Let $\tau^* : T^*M \rightarrow M$ denote the bundle projection.
Let $\mathcal{P}(H)$ be the set of 1-periodic orbits of the corresponding Hamiltonian vector field. We assume that every $x\in \mathcal{P}(H)$ is non-degenerate. Let $J$ be a 1-periodic $\omega$-compatible almost complex structure on $T^*M$, which is assumed to be generic, so that for each pair $x^-,x^+\in \mathcal{P}(H)$ the space $\mathcal{M}(x^-.x^+)$ of solutions of the Floer equation
\begin{equation}
\label{floer}
\partial_s u - J(t,u) \bigl(\partial_t u - X_H(t,u) \bigr) = 0,
\end{equation}
which are asymptotic to $x^-$ and $x^+$ for $s \rightarrow -\infty$ and $s\rightarrow +\infty$ is the zero-set of a Fredholm section of a Banach bundle which is transverse to the zero-section. Let us recall how the finite dimensional manifold $\mathcal{M}(x^-,x^+)$ is oriented. 

We fix a coherent orientation for $\Sigma$.
For every $x\in \mathcal{P}(H)$, we fix once and for all a unitary trivialization (with respect to $J$)
\[
\Phi_x : \T \times \C^n \cong x^*(TT^*M)
\]
which is vertical-preserving, meaning that the image of $\T \times \lambda_0$ is the vertical subbundle $x^*(T^v T^*M)$, and such that $\Phi_x|_{\T\times \lambda_0}$ is orientation-preserving. Such trivializations are easily built starting from an orientation-preserving trivialization of $(\tau^* \circ x)^*(TM)$ (recall that $M$ is assumed to be oriented).

For every $u\in \mathcal{M}(x,y)$, we can find a unitary trivialization $\Phi_u$ of $u^*(TT^*M)$ which agrees with $\Phi_x$ and $\Phi_y$ at $\{-\infty\} \times \T$ and $\{+\infty\} \times \T$ (the reason is that every loop in $\mathrm{SO}(n)$ is contractible within $\mathrm{U}(n)$, see \cite[Lemma 1.7]{as06}). Such a trivialization need not be vertical preserving. By using the trivialization $\Phi_u$, the linearization of the Floer equation (\ref{floer}) along $u$ is conjugated to a Fredholm operator $D_S$ in $\Sigma(S_x,S_y)$, where the asymptotic loops $S_x$ and $S_y$ depend on the fixed trivializations $\Phi_x$ and $\Phi_y$. Therefore, the orientation $o(S_x,S_y)$ of $\det(\Sigma(S_x,S_y))$ induces an orientation of $\ker D_S \cong T_u \mathcal{M}(x,y)$. By \cite[Lemma 13]{fh93}, this orientation does not depend on the choice of $\Phi_u$ and defines an orientation of $\mathcal{M}(x,y)$. 

In particular, when $\mu_{CZ}(x)-\mu_{CZ}(y)=1$, $\mathcal{M}(x,y)$ is an oriented one-dimensional manifold. Denoting by $[u]$ the equivalence class of $u$ in the zero-dimensional manifold $\mathcal{M}(x,y)/\R$, the quotient of $\mathcal{M}(x,y)$ by the action of $\R$ by translation, we define
\[
\epsilon([u]) \in \{-1,+1\}
\]
to be $+1$ if the $\R$-action is orientation preserving on the connected component of $u$ in $\mathcal{M}(x,y)$, $-1$ otherwise.

If we also assume that $H$ has quadratic growth on the fibers of $T^*M$ (as in assumptions (H1) and (H2) of \cite{as06}) and that $J$ is $C^0$-close enough to a Levi-Civita almost complex structure, compactness holds and $\mathcal{M}(x,y)/\R$ is a finite set whenever $\mu_{CZ}(x)-\mu_{CZ}(y)=1$. The integers
\[
n(x,y) := \sum_{[u] \in \mathcal{M}(x,y)/\R} \epsilon([u])
\]
are the coefficients of the boundary operator
\[
\partial_k = \partial_k (H,J) : CF_k (H) \rightarrow CF_{k-1}(H), \qquad \partial_k x := \sum_{\substack{y\in \mathcal{P}(H) \\ \mu_{CZ}(y) = k - 1 }} n(x,y)\, y, 
\]
of the standard Floer complex of $(H,J)$. Here, $CF_k(H)$ denotes the free Abelian group generated by periodic orbits of Conley-Zehnder index $k$.

Let $u\in \mathcal{M}(x,y)$. The vertical preserving unitary trivialization $\Phi_{x}$ of $x(TT^*M)$ can be continued to a vertical preserving unitary trivialization $\Psi_u$ of $u^*(TT^*M)$ over $\overline{\R} \times \T$. By restriction to 
\[
y^*(T^v T^*M) \cong (\tau^* \circ y)^* (T^*M),
\]
the composition $\Phi_y^{-1} \circ \Psi_u(+\infty,\cdot)$ defines a loop $U$ in $\mathrm{SO}(n)$. We define $\delta([u])$ to be $+1$ if either $n=1$, or $n=2$ and the homotopy class of $U$ in $\mathrm{SO}(2)$ is an even multiple of the generator of $\pi_1(\mathrm{SO}(2))=\Z$, or $n\geq 3$ and $U$ is contractible in $\mathrm{SO}(n)$ (equivalently and without making a case distinction on the dimension $n$, if $U$ lifts to a closed loop in $\mathrm{Spin}(n)$). Otherwise, we define $\delta([u])$ to be $-1$. It is easy to check that this definition does not depend on the choice of $\Psi_u$. Moreover, it depends only on the homotopy class of $u$ among the cylinders with asymptotics $x$ and $y$. Furthermore, if the pair $(u,v)$ belongs to $\mathcal{M}(x,y) \times \mathcal{M}(y,z)$ and $(u',v')\in \mathcal{M}(x,y') \times \mathcal{M}(y',z)$ is the pair (unique up to translations) which is obtained as right boundary of the component of $\mathcal{M}(x,z)$ having $(u,v)$ as left boundary, we have
\begin{equation}
\label{cpr}
\delta([u]) \delta([v]) = \delta([u']) \delta([v']).
\end{equation}
The coefficients of the twisted Floer complex of $(H,J)$ are defined as
\[
\widehat{n}(x,y) := \sum_{[u]\in \mathcal{M}(x,y)/\R} \epsilon([u]) \delta([u]).
\]
The identity (\ref{cpr}) implies that also the operator
\[
\widehat\partial_k =  \widehat\partial_k (H,J) : CF_k (H) \rightarrow CF_{k-1}(H), \qquad \widehat\partial_k x := \sum_{\substack{y\in \mathcal{P}(H) \\ \mu_{CZ}(y) = k - 1 }} \widehat{n}(x,y)\, y, 
\]
is a boundary. The resulting complex $(CF_*(H),\widehat{\partial}_*)$ is the twisted Floer complex of $(H,J)$. Both the standard and the twisted Floer complex change by a chain isomorphism when the orientation data (that is the coherent orientation for $\Sigma$ and the choice of the trivializations $\Phi_x$, for $x$ in $\mathcal{P}(H)$) are changed.

If the manifold $M$ is Spin, meaning that the second Stiefel-Whitney class \linebreak
$w_2(TM)\in H^2(M;\Z_2)$ of its tangent bundle vanishes, we can choose the vertical-preserving trivializations $\Phi_x$ in such a way that $\delta([u])$ is always $+1$. More generally, this is true if $f^*(w_2(TM))=0$ for every continuous map $f:\T^2 \mapsto M$. An example of a manifold which is not Spin but for which the latter condition holds is the non-trivial $S^2$-bundle over the orientable closed surface of genus 2 (we are grateful to B.\ Martelli for suggesting us this example).

Indeed, the construction goes as follows: we fix an orthogonal and orientation-preserving trivialization of $q_0^*(TM)$ for a loop $q_0:\T \rightarrow M$ in each free homotopy class. If $q:\T \rightarrow M$ is a loop which is freely homotopic to $q_0$ and $w_q:[0,1]\times \T \rightarrow M$ is a homotopy between $q_0$ and $q$, we transport the trivialization of $q_0^*(TM)$ along $w_q$ and get a trivialization $\Psi_q$ of $q^*(TM)$. For every $x\in \mathcal{P}(H)$ we choose $\Phi_x$ to be a vertical-preserving unitary trivialization of $x^*(TT^*M)$ such that the induced trivialization of $(\tau^*\circ x)^*(T^*M) \cong (\tau^*\circ x)^*(TM)$ is homotopic to $\Psi_{\tau^* \circ x}$. We claim that with these choices, if $u:\overline{\R} \times \T \rightarrow T^*M$ is a cylinder which connects two periodic orbits $x$ and $y$, then $\delta([u])=+1$. Indeed, we recall that if $E$ is an $n$-dimensional oriented Riemannian vector bundle over $\T^2$ and $\Phi$ is an orthogonal and orientation preserving trivialization of the restriction of $E$ to a circle $\T\times \{\mathrm{pt}\}$, transportation around the torus defines another trivialization $\Psi$ of the same restriction and the loop $\Psi^{-1}\circ \Phi: \T \rightarrow \mathrm{SO}(n)$ lifts to a closed loop in $\mathrm{Spin}(n)$ if and only if $w_2(E)=0$.
Then our assertion follows from the fact that $w_2(f^*(TM))=0$, where $f:\T^2 \rightarrow M$ is the torus which is obtained  by gluing the three cylinders $w_{\tau^* \circ x}$, $\tau^* \circ u$ and $w_{\tau^* \circ y}(-\cdot,\cdot)$.

With the above choices, the twisted Floer complex coincides with the standard one. Since the homology of both Floer complexes does not depend on the choice of the trivializations $\Phi_x$, we conclude that when $w_2(TM)$ vanishes on 2-tori the twisted Floer homology of $T^*M$ coincides with the standard one.
  
\section{The chain isomorphism $\Theta$}
\label{tre}

In this section we recall the definition of the spaces of half-cylinders which determine the isomorphism $\Theta$, together with the construction of their orientations, and we explain why $\Theta$ is indeed a chain isomorphism from the Morse complex of the Lagrangian action functional to the twisted Floer complex. 

If $S_1,S_2 \in C^0(\T,\mathrm{Sym}(2n))$ are two non-degenerate loops,  orientations $o(S_1)$ of $\det(\Sigma^+(S_1))$ and $o(S_1,S_2)$ of $\det (\Sigma(S_1,S_2))$ induce an orientation
\[
o(S_1) \# o(S_1,S_2)
\]
of $\det(\Sigma^+(S_2))$. This construction is associative. Fix a coherent orientation for $\Sigma$. A compatible orientation for $\Sigma^+$ is the choice of an orientation $o(S)$ of the line bundle $\det(\Sigma^+(S))$, for every non-degenerate loop $S$, such that for every pair $(S_1,S_2)$ there holds
\[
o(S_2) = o(S_1) \# o(S_1,S_2).
\]
There exist exactly two orientations for $\Sigma^+$ which are compatible with a given coherent orientation for $\Sigma$.

Assume that the Hamiltonian $H$, which is still assumed to satisfy the assumption of the previous section, is also uniformly fiberwise convex, meaning that
\[
d_{pp} H(t,q,p) \geq h_0 I, \qquad \forall (t,q,p)\in \T \times T^*M,
\]
for some $h_0>0$. In this case, Legendre duality defines a Lagrangian $L\in C^{\infty}(\T \times TM)$, whose corresponding action functional
\[
\mathcal{E}(q) := \int_{\T} L\bigl(t,q(t),q'(t)\bigr) \, dt
\]
has a well-defined Morse complex 
\[
\partial_* = \partial_*(\mathcal{E},Y) : CM_*(\mathcal{E}) \longrightarrow CM_{*-1}(\mathcal{E})
\] 
on the Hilbert manifold $\Lambda^1(M)$ of closed loops in $M$ of Sobolev class $W^{1,2}$. Here $Y$ is a smooth negative pseudo-gradient vector field for $\mathcal{E}$ with the Morse-Smale property (indeed, under the above assumptions on $L$, the functional $\mathcal{E}$ is just of class $C^{1,1}$ on $\Lambda^1(M)$, so it is more convenient to use a smooth pseudo-gradient vector field, rather than the gradient vector field with respect to a Riemannian metric, which would be just Lipschitz continuous; this regularity issue had been overlooked in \cite{as06}, but has been already corrected in \cite{as09b}). 

Denote by $\mathcal{P}(L)$ the set of critical points of $\mathcal{E}$, which coincides with the set of all loops $\tau^*\circ x$ for $x\in \mathcal{P}(H)$. If $q\in \mathcal{P}(L)$ and $x\in \mathcal{P}(H)$, the space $\mathcal{M}^+(q,x)$ is the set of all solutions $u: [0,+\infty[ \times \T \rightarrow T^*M$ of the Floer equation (\ref{floer}) such that $u(s,\cdot)$ converges to $x$ for $s\rightarrow +\infty$ and
the loop $\tau^* \circ u(0,\cdot)$ belongs to the unstable manifold $W^u(q)$ of $q$ with respect to the flow of $Y$. 

For a generic choice of the almost complex structure which appears in the Floer equation, $\mathcal{M}^+(q,x)$ is the zero-set of a section 
\[
\partial_{J,H}^+ : \mathcal{B}^+(q,x) \rightarrow \mathcal{W}^+(q,x)
\]
of a Banach bundle which is transverse to the zero-section, and it is a manifold of dimension $m(q) - \mu_{CZ}(x)$, where $m(q)$ denotes the Morse index of the critical point $q$ of $\mathcal{E}$ (see \cite[Section 3.1]{as06}).

Let us fix a coherent orientation for $\Sigma$ and a vertical preserving unitary trivialization of $x^*(TT^*M)$ for every $x\in \mathcal{P}(H)$ as in the previous section. Let us fix also an orientation for $\Sigma^+$ which is compatible with the coherent orientation for $\Sigma$ and an orientation of the unstable manifold of every critical point of $\mathcal{E}$. 

These data determine an orientation of $\mathcal{M}^+(q,x)$, for every pair $(q,x) \in \mathcal{P}(L) \times \mathcal{P}(H)$. In fact, if $u\in \mathcal{M}^+(q,x)$, the bounded linear operator
\[
T_u \mathcal{B}^+(q,x) \rightarrow T_{\tau^* \circ u(0,\cdot)} W^u(q), \qquad w \mapsto D\tau^* (u(0,\cdot)) [w(0,\cdot)]
\]
is surjective, has a kernel $W_u$ of codimension $m(q)$, and its codomain is oriented by the orientation of  the unstable manifold $W^u(q)$. Therefore, the quotient 
\[
T_u \mathcal{B}^+(q,x) /W_u
\]
is oriented, and so is the line
\[
\Lambda^{\max} \bigl( T_u \mathcal{B}^+(q,x) /W_u \bigr).
\]
Let $\Phi_u$ be a vertical-preserving unitary trivialization of $u^*(TT^*M)$ over $[0,+\infty] \times \T$ which agrees with $\Phi_x$ at $\{+\infty\}\times \T$. This trivialization conjugates the restriction to $W_u$ of the fiberwise differential at $u$ of the section $\partial_{J,H}^+$ to an operator $D_S^+$ in $\Sigma^+(S_x)$. Therefore, 
\[
\det \bigl( D_f \partial_{J,H}^+ (u)|_{W_u} \bigr) \cong \det\bigl(D_S^+\bigr)
\]
inherits an orientation from $o(S_x)$. Proposition \ref{easy} implies that this orientation does not depend on the choice of the vertical-preserving trivialization $\Phi_u$. Indeed, if $\Psi_u$ is another vertical-preserving trivialization, the transition map $U = \Psi_u^{-1} \circ \Phi_u$ takes values into $\mathrm{SO}(n)$ and is such that $U(+\infty,t)=I$. 

From the canonical isomorphism
\[
\det \bigl( D_f \partial_{J,H}^+ (u) \bigr) \cong \det \bigl( D_f \partial_{J,H}^+ (u)|_{W_u} \bigr) \otimes \Lambda^{\max} \bigl( T_u \mathcal{B}^+(q,x) /W_u \bigr),
\]
we get the required orientation of 
\[
\det \bigl( D_f \partial_{J,H}^+ (u) \bigr) = \Lambda^{\max} \bigl( T_u \mathcal{M}^+(q,x) \bigr),
\]
so $\mathcal{M}^+(q,x)$ is oriented. 

In particular, when $m(q) = \mu_{CZ}(x)$, the zero-dimensional manifold $\mathcal{M}^+(q,x)$ is oriented, meaning that each point $u\in \mathcal{M}^+(q,x)$ is given a number $\epsilon^+(u)\in \{-1,+1\}$. In this case, the integer $n^+(q,x)$ is defined as
\[
n^+(q,x) := \sum_{u\in \mathcal{M}^+(q,x)} \epsilon^+(u).
\]
The homomorphism 
\[
\Theta : CM_*( \mathcal{E})  \longrightarrow CF_*(H)
\]
is defined generator-wise as
\[
\Theta q := \sum_{\substack{x\in \mathcal{P}(H)\\ \mu_{CZ}(x) = m(q)}} n^+(q,x)\, x, \qquad \forall q\in \mathcal{P}(L).
\]
The correct formulation of the periodic part of Theorem 3.1 of \cite{as06} is:

\begin{thm}
The map $\Theta$ is a chain isomorphism from the Morse complex 
\[
\bigl\{CM_*(\mathcal{E}), \partial_*(\mathcal{E},Y)\bigr\}
\] 
of the Lagrangian action functional associated to $L$ to the twisted Floer complex 
\[
\bigl\{CF_*(H), \widehat{\partial}_*(H,J)\bigr\}
\]
of the dual Hamiltonian $H$.
\end{thm}

Let us explain why $\Theta$ is a chain map. Let $q\in \mathcal{P}(L)$ and $y\in \mathcal{P}(H)$, with $\mu_{CZ}(y) = \mu(q)-1$. The coefficient of $y$ in $\widehat\partial \Theta q$ is the sum
\begin{equation}
\label{dete}
\sum_{\substack{x\in \mathcal{P}(H) \\ \mu_{CZ}(x) = m(q)}} n^+(q,x) \widehat{n}(x,y),
\end{equation}
while its coefficient in $\Theta \partial q$ is the sum
\begin{equation}
\label{tede}
\sum_{\substack{r\in \mathcal{P}(L) \\ m(r) = m(q)-1}} n_M(q,r) n^+(x,y),
\end{equation}
where $n_M(r,q)$ are the coefficients of the Morse complex. In order to show that the numbers (\ref{dete}) and (\ref{tede}) coincide, one has to analyze two different gluing situations.

In the first situation, we have $x\in \mathcal{P}(H)$ with $\mu_{CZ}(x) = m(q)$, $u\in \mathcal{M}^+(q,x)$ and $v\in \mathcal{M}(x,y)$. The pair $(u,v)$ contributes to the sum (\ref{dete}) by either $+1$ or $-1$.
Let $W$ be the connected component of the one-dimensional manifold $\mathcal{M}^+(q,y)$ having the pair $(u,v)$ as one of the two limiting points. Let $\Phi_u$ be a vertical-preserving unitary trivialization of $u^*(TT^*M)$ which agrees with $\Phi_x$ at $\{+\infty\} \times \T$ and let $\Phi_v$ be a unitary trivialization of $v^*(TT^*M)$ which agrees with $\Phi_x$ and $\Phi_y$ at $\{-\infty\} \times \T$ and  $\{+\infty\} \times \T$. Assume first that also $\Phi_v$ can be chosen to be vertical-preserving. In this case, $\delta([v])=+1$. Moreover, the trivializations $\Phi_u$ and $\Phi_v$ can be glued and then slightly perturbed in order to produce a vertical preserving trivialization $\Phi_w$ of $w^*(TT^*M)$, where $w$ is an element of $W$ close to the limiting point $(u,v)$. The trivialization of $\Phi_w$ is an admissible one in the definition of the orientation of $W$, so we deduce that in this case the orientation of $W$ is compatible with the orientation $\epsilon^+(u) \epsilon([v])$ of its limiting point $(u,v)$. Consider now the general case, in which it might be impossible to choose $\Phi_v$ to be vertical-preserving and with the given asymptotics. Changing the trivialization which is obtained by gluing $\Phi_u$ and $\Phi_v$ into a vertical-preserving one involves multiplication by a map $U:[0,+\infty]\times \T \rightarrow \mathrm{U}(n)$ such that $U(0,t)\in \mathrm{SO}(n)$ and $U(+\infty,t)=I$, for every $t\in \T$. By Proposition \ref{complex}, in this case the orientation of $W$ is compatible with the orientation $\epsilon^+(u) \epsilon([v])$ of its limiting point $(u,v)$ if and only if $\delta([v])=1$. We conclude that in every case the orientation of $W$ is compatible with the orientation $\epsilon^+(u) \epsilon([v]) \delta([v])$ of its limiting point $(u,v)$.

The second gluing situation arises on the Morse side and presents no difficulties: we have $r\in \mathcal{P}(L)$ with $m(r) = m(q)-1$, an orbit $\gamma$ in $W^u(q)\cap W^s(r)$ and an element $u\in \mathcal{M}^+(r,y)$. If $W$ is the connected component of the one-dimensional manifold $\mathcal{M}^+(q,y)$ having the pair $(\gamma,u)$ as one of the two limiting points, then the orientation of $W$ is compatible with the orientation of its limiting point $(\gamma,u)$. The standard cobordism argument implies that the two numbers (\ref{dete}) and (\ref{tede}) coincide, and hence $\Theta$ is a chain map.

\section{Proof of Proposition \ref{complex}}
\label{propro}

Let us start with the following lemma, whose proof is similar to that of Lemma 13 in \cite{fh93}:

\begin{lem}
\label{contr}
Assume moreover that the loop $U(0,\cdot)$ is contractible in $\mathrm{SO}(n)$. Then the canonical lift to the determinant bundle of the map (\ref{mappa}) is 
orientation-preserving.
\end{lem}

\begin{proof}
Since $U(0,\cdot)$ is contractible in $\mathrm{SO}(n)$ and $\pi_2(\mathrm{U}(n))=0$, we can find a homotopy
\[
r\mapsto U_r \in C^{\infty}( [0,+\infty] \times \T, \mathrm{U}(n) ) , \quad r\in [0,1],
\]
such that
\[
U_0 = U, \qquad U_1(s,t) = V(s) \quad \mbox{with} \quad V(0)=V(+\infty)=I, \quad \forall (s,t)\in [0,+\infty] \times \T.
\]
Since $\pi_1(\mathrm{U}(n))=\Z$ is generated by the homotopy class of the loop 
\[
\theta \mapsto \bigl( e^{2\pi \theta i} I_{\C}\bigr) \oplus I_{\C^{n-1}}, \qquad \theta\in \T,
\]
we can also assume that
\begin{equation}
\label{formaV}
V(s) = \bigl( e^{2\pi \theta(s) i} I_{\C}\bigr) \oplus I_{\C^{n-1}},
\end{equation}
where $\theta$ is a smooth real valued function on $[0,+\infty]$ such that $\theta(0)$ and $\theta(+\infty)$ are integers.
By using such a homotopy, it is enough to show that the lift to the determinant bundle of the map
\[
\Sigma^+(S^+) \rightarrow \Sigma^+(S^+), \qquad
D_S^+ \mapsto V D_S^+ V^{-1}
\]
is orientation-preserving.

By gluing, it is enough to check this fact for a particular $S^+$, for instance 
\[
S^+(t) = -\pi I, \qquad \forall t\in \T.
\]
An element of $\Sigma^+(-\pi I)$ is the operator with constant coefficients $D_{-\pi I}^+$ and it is enough to check that the canonical map
\[
\det (D^+_{-\pi I}) \longrightarrow \det ( V D^+_{-\pi I} V^{-1})
\]
is orientation-preserving. By the form (\ref{formaV}) of $V$, we may assume that $n=1$ and
\[
V(s) = e^{2\pi i \theta(s)}, \qquad \forall s\in [0,+\infty].
\]
In this case, the operator $D_{-\pi I}^+$ is surjective and its kernel is 
\[
\ker D_{-\pi I}^+ = \mathrm{Span}_{\R} (i e^{-\pi s}).
\]
See Claim 2 in the proof of \cite[Theorem 3.4]{as06}. A projector onto this one-dimensional space is the operator 
\begin{eqnarray*}
& P : W^{1,p}_{\lambda_0}(]0,+\infty[\times \T,\C) \rightarrow W^{1,p}_{\lambda_0}(]0,+\infty[\times \T,\C), & \\ & (Pu)(s,t) := 2\pi  \left( \displaystyle{\iint_{[0,+\infty[\times \T}} e^{-\pi \sigma} \im u(\sigma,\tau) \, d\sigma\, d\tau \right) i e^{-\pi s}. &
\end{eqnarray*}
By considering the homotopy
\[
V_r(s) := V(s/r) = e^{2\pi i \theta(s/r)}, \qquad r\in [1,+\infty[,
\]
it is enough to check that the canonical map
\[
\det (D^+_{-\pi I}) = \Lambda^1 (\ker D^+_{-\pi I}) \longrightarrow \det ( V_r D^+_{-\pi I} V^{-1}_r) = \Lambda^1 (\ker V_r D^+_{-\pi I} V_r^{-1})
\]
is orientation-preserving when $r$ is large. By (\ref{ide}), the family of operators
\[
V_r D^+_{-\pi I} V^{-1}_r = \partial_s - J_0 \partial_t - \frac{2\pi i}{r} \theta'(s/r) + \pi I
\]
converges to $D^+_{-\pi I}$ in the operator norm for $r\rightarrow +\infty$. Therefore, the restriction
\[
P|_{\ker V_r D^+_{-\pi I} V^{-1}_r} : \ker V_r D^+_{-\pi I} V^{-1}_r \rightarrow  \ker D^+_{-\pi I} 
\] 
is an isomorphism for $r$ large enough and it induces an orientation-preserving map between $\Lambda^1$ of these spaces. Thus, it is enough to check that for $r$ large the composition $P\circ V_r$ of the maps
\[
\ker D^+_{-\pi I} \rightarrow \ker V_r D_{-\pi I}^+ V_r^{-1}, \quad u \mapsto V_r u=e^{2\pi i \theta(\cdot/r)} u,
\]
and
\[
 \ker V_r D^+_{-\pi I} V^{-1}_r \rightarrow  \ker D^+_{-\pi I}, \quad u \mapsto Pu,
\]
is orientation-preserving. Such a composition maps the generator $i e^{-\pi s}$ into
\[
2\pi  \left( \int_0^{+\infty} e^{-2\pi \sigma} \bigl( \cos (2\pi \theta(\sigma/r)) - \sin (2\pi \theta(\sigma/r))\bigr) \, d\sigma \right) i e^{-\pi s}.
\]
By dominated convergence, the above integral tends to $1/2\pi$ for $r\rightarrow +\infty$, so for $r$ large enough the generator $i e^{-\pi s}$ is mapped into a positive multiple of itself, which proves that the above composition is orientation-preserving.
\end{proof}

Let $n=2$ and let us identify $\R^4$ with $\C^2$. From the proof of Theorem 3.4 in \cite{as06}, in particular Claim 2, we know that the operator $D_{-\pi I}^+$ is surjective and has the 2-dimensional kernel 
\[
\set{ e^{-\pi s} i z}{z\in \R^2}.
\]
Let $\varphi\in C^{\infty}(\R)$ be a non-decreasing function such that $\varphi(s)=0$ for $s\leq 0$ and $\varphi(s)=1$ for $s\geq 1/2$. 
Set, for $0\leq s \leq 1/2$, 
\[
W(s,t) := \omega(s) \left( \begin{array}{cc} \bigl( 1 - \varphi(s) \bigr) \cos (2\pi t) - \varphi(s) i & - \bigl( 1 - \varphi(s) \bigr) \sin (2\pi t) \\ \bigl( 1 - \varphi(s) \bigr) \sin (2\pi t) & \bigl( 1 - \varphi(s) \bigr) \cos (2\pi t) + \varphi(s) i \end{array} \right), 
\]
where
\[
\omega(s):= \frac{1}{\sqrt{ \varphi(s)^2 + \bigl( 1 - \varphi(s) \bigr)^2 }},
\]
and, for $s\geq 1/2$,
\[
W(s,t) := \left( \begin{array}{cc} - i e^{i \frac{\pi}{2} \varphi(s-\frac{1}{2})} & 0 \\ 0 & i e^{-i \frac{\pi}{2} \varphi(s-\frac{1}{2})} \end{array} \right).
\]
It is easy to check that
\[
W \in C^{\infty}( [0,+\infty] \times \T, \mathrm{U}(2)),
\]
and
\[
W(0,t) = \left( \begin{array}{cc}  \cos (2\pi t) & -  \sin (2\pi t) \\  \sin (2\pi t) & \cos (2\pi t)  \end{array} \right), \qquad W(s,t) = I \quad \forall s\geq 1,
\]
for every $t\in \T$.

We also set, for $r\in [0,1]$,
\[
W_r(s,t) := W(r+s,t), \quad \mbox{so that } W_0 = W \mbox{ and } W_1 = I,
\]
and
\[
T_r := (\partial_s W_r) W_r^{-1} - i (\partial_t W_r) W_r^{-1} - \pi I.
\]
We consider the path
\[
[0,1] \rightarrow \Sigma^+(-\pi I), \qquad r\mapsto D^+_{T_r},
\]
in dimension $n=2$. 

By the identity (\ref{ide}),
\begin{equation}
\label{ide2}
W_r ( \partial_s - i \partial_t + \pi I) W_r^{-1} = \partial_s - i \partial_t - T_r.
\end{equation}
In particular, for $r=0$ there holds
\[
D_{T_0}^+ = W D_{-\pi I}^+ W^{-1},
\]
while for $r=1$
\[
D_{T_1}^+ = D_{-\pi I}^+.
\]
Since  the multiplication by $W_r^{-1}$ need not preserve the boundary condition $u(0,\cdot)\in \lambda_0 = i \R^2$, for an arbitrary $r\in [0,1]$ the operator $D_{T_r}^+$ is not related to $D_{-\pi I}^+$ by conjugacy. However, (\ref{ide2}) implies that
\begin{eqnarray*}
\ker D_{T_r}^+ = \bigl\{ W_r u\; \big| &&  u\in W^{1,p}(]0,+\infty[\times \T,\C^2), \; \partial_s u - i \partial_t u + \pi u = 0, \\ && \re W_r(0,t) u(0,t) = 0 \; \forall t\in \T \bigr\}.
\end{eqnarray*}
A $\C^2$-valued function $u$ solves the equation
\[
\partial_s u - i \partial_t u + \pi u = 0
\]
on $[0,+\infty[ \times \T$ if and only if
\[
u(s,t) = \sum_{k\in \Z} e^{2\pi i k t} e^{-\pi(2k+1)s} u_k, \qquad \mbox{with } u_k\in \C^2.
\]
If such a function is in $W^{1,p}(]0,+\infty[\times \T,\C^2)$, then all the coefficients $u_k$ with $k<0$ vanish. Therefore,
\begin{equation}
\label{nucleo}
\begin{split}
\ker D_{T_r}^+ = \bigl\{ W_r u\; \big| \; &  u(s,t) = \sum_{k\geq 0} e^{2\pi i k t} e^{-\pi (2k+1)s} u_k, \mbox{ where } (u_k)_{k\in \N} \subset \C^2 \\ & \mbox{ is such that }  u\in W^{1,p}(]0,+\infty[\times \T,\C^2), \\ & \mbox{ and }
\re W(r,t) \sum_{k\geq 0} e^{2\pi k i t} u_k = 0 \; \forall t\in \T \bigr\}.
\end{split}
\end{equation}
In the following two lemmas we use the above identity to compute the kernel of $D^+_{T_r}$ 
for $0\leq r \leq 1/2$ and for $1/2\leq r \leq 1$.

\begin{lem}
\label{kernel1}
If $0\leq r \leq 1/2$ then $D_{T_r}^+$ is onto and its kernel is the two-dimensional space which is generated  by the pair of functions
\begin{eqnarray*}
u_r(s,t) & := & W_r(s,t) \Bigl( \hat{u}_0(r) e^{-\pi s} + \hat{u}_1 (r) e^{2\pi i t} e^{-3\pi s} \Bigr), \\
v_r(s,t) & := & W_r(s,t) \Bigl( \hat{v}_0(r) e^{-\pi s} + \hat{v}_1 (r) e^{2\pi i t} e^{-3\pi s} \Bigr),
\end{eqnarray*}
where
\[
\begin{array}{rlrl}
\hat{u}_0(r) :=& \varphi(r)^2 \binom{1}{-1} + (1-\varphi(r))^2 \binom{i}{i}, & 
\hat{v}_0(r) :=& \varphi(r)^2 \binom{1}{1} + (1-\varphi(r))^2 \binom{-i}{i} , \\
\hat{u}_1(r) :=& \varphi(r)(1-\varphi(r)) \binom{ -1-i }{ 1-i }, &
\hat{v}_1(r) :=& \varphi(r)(1-\varphi(r)) \binom{1-i }{ 1+i}.
\end{array}
\]
\end{lem}

\begin{proof}
Set $u_k=(x_k,y_k)$ with $x_k,y_k\in \C$. Since $0\leq r \leq 1/2$, the vector
\begin{equation}
\label{vect}
W(r,t) \sum_{k\geq 0} e^{2\pi i k t} u_k 
\end{equation}
equals
\[
\omega(r) \sum_{k\geq 0} e^{2\pi i k t} \left( \begin{array}{c} \bigl( (1-\varphi) \cos (2\pi t) - \varphi i \bigr) x_k - (1-\varphi) \sin (2\pi t) y_k \\ (1-\varphi) \sin(2\pi t) x_k + \bigl( (1-\varphi) \cos (2\pi t) + \varphi i \bigr) y_k \end{array} \right).
\]
Here and in the following equations $\varphi$ is evaluated at $r$. Therefore, the real part of (\ref{vect}) vanishes if and only if the following two equations hold:
\begin{eqnarray*}
\begin{aligned}
 \sum_{k\geq 0} & \Bigl( \cos (2\pi kt) \bigl( (1-\varphi) \cos (2\pi t) \re x_k + \varphi \im x_k - (1-\varphi) \sin (2\pi t) \re y_k \bigr)  \\
&  - \sin (2\pi k t) \bigl( (1-\varphi) \cos (2\pi t) \im x_k - \varphi \re x_k - (1-\varphi) \sin (2\pi t) \im y_k \bigr) \Bigr) = 0,  \\
\end{aligned}
\\
\begin{aligned}
\sum_{k\geq 0} & \Bigl( \cos (2\pi kt) \bigl( (1-\varphi) \sin (2\pi t) \re x_k + (1-\varphi) \cos (2\pi t) \re y_k - \varphi \im y_k \bigr)\\
& - \sin (2\pi k t) \bigl( (1-\varphi) \sin (2\pi t) \im x_k + (1-\varphi) \cos (2\pi t) \im y_k + \varphi \re y_k \bigr) \Bigr) = 0.
\end{aligned}
\end{eqnarray*}
By using the identities
\begin{eqnarray*}
\cos (2\pi k t) \cos (2\pi t) &=& \frac{1}{2} \bigl( \cos (2\pi(k+1)t) + \cos (2\pi (k-1)t) \bigr), \\ \cos (2\pi k t) \sin (2\pi t) &=& \frac{1}{2} \bigl( \sin (2\pi(k+1)t) - \sin (2\pi (k-1)t) \bigr), \\
\sin (2\pi k t) \cos (2\pi t) &=& \frac{1}{2} \bigl( \sin (2\pi(k+1)t) + \sin (2\pi (k-1)t) \bigr),
\\ \sin (2\pi k t) \sin (2\pi t) &=& \frac{1}{2} \bigl( - \cos (2\pi(k+1)t) + \cos (2\pi (k-1)t )\bigr),  
\end{eqnarray*}
and the fact that the set
\[
\bigl\{ \cos (2\pi k t) \bigr\}_{k\geq 0} \cup \bigl\{\sin (2\pi k t) \bigr\}_{k\geq 1}
\]
is an orthogonal family in $L^2(\T)$, we find that the previous two equations are equivalent to the following infinite system:
\begin{equation}
\label{s}
\begin{split}
\textstyle{\frac{1-\varphi}{2}} \re x_1 + \varphi \im x_0 + \frac{1-\varphi}{2} \im y_1 &= 0, 
\\
\frac{1-\varphi}{2} \re y_1 - \varphi \im y_0 - \frac{1-\varphi}{2} \im x_1 &= 0,  \\
\textstyle{\frac{1-\varphi}{2}} (2 \re x_0 + \re x_2) + \varphi \im x_1 + \textstyle{\frac{1-\varphi}{2}} \im y_2 &=0, \\
\textstyle{\frac{1-\varphi}{2}} (2\re y_0 + \re y_2) - \varphi \im y_1 - \textstyle{\frac{1-\varphi}{2}} \im x_2 &= 0, \\
- \textstyle{\frac{1-\varphi}{2}} (2\re y_0 - \re y_2) - \textstyle{\frac{1-\varphi}{2}} \im x_2 + \varphi \re x_1 &= 0, \\
\textstyle{\frac{1-\varphi}{2}} (2\re x_0 - \re x_2) - \textstyle{\frac{1-\varphi}{2}} \im y_2 - \varphi \re y_1 &= 0, \\ 
\textstyle{\frac{1-\varphi}{2}} (\re x_{k-1} + \re x_{k+1}) + \varphi \im x_k - \frac{1-\varphi}{2} (\im y_{k-1} -  \im y_{k+1}) &= 0,  
\\
\textstyle{\frac{1-\varphi}{2}} (\re y_{k-1} + \re y_{k+1} ) - \varphi \im y_k - \frac{1-\varphi}{2} (-\im x_{k-1} + \im x_{k+1} ) &= 0, 
\\
 -\textstyle{\frac{1-\varphi}{2}} ( \re y_{k-1} -  \re y_{k+1})  - \frac{1-\varphi}{2} ( \im x_{k-1} +    \im x_{k+1}) + \varphi \re x_k &= 0,  
\\
\textstyle{\frac{1-\varphi}{2}} (\re x_{k-1} - \re x_{k+1} ) - \frac{1-\varphi}{2} (\im y_{k-1} + \im y_{k+1} ) - \varphi \re y_k &= 0, 
\end{split}
\end{equation}
for all $k\geq 2$. In order to solve this system, it is convenient to consider the change of variables
\[
\xi_k = \re x_k - \im y_k, \quad \lambda_k = \re x_k + \im y_k, \quad \eta_k = \re y_k + \im x_k, \quad \mu_k = \re y_k - \im x_k,
\]
whose inverse is
\[
\re x_k = \textstyle{\frac{1}{2}} (\xi_k+\lambda_k), \;\; \im x_k = \frac{1}{2} (\eta_k - \mu_k), \;\; \re y_k = \frac{1}{2}(\eta_k + \mu_k), \;\; \im y_k = \frac{1}{2} (\lambda_k - \xi_k).
\]
The system (\ref{s}) can now be rewritten as
\begin{eqnarray}
\label{t0a}
(1-\varphi) \lambda_1 + \varphi (\eta_0 - \mu_0) &=& 0, \\
\label{t0b}
(1-\varphi) \mu_1 + \varphi (\xi_0 - \lambda_0) &=& 0, \\
\label{t1a}
(1-\varphi) (\xi_0+\lambda_0) + (1-\varphi) \lambda_2 +  \varphi (\eta_1-\mu_1) &=& 0, \\
\label{t1b}
(1-\varphi) (\eta_0+\mu_0) + (1-\varphi) \mu_2 - \varphi (\lambda_1 - \xi_1) &=& 0, \\
\label{t1c}
-(1-\varphi) (\eta_0+\mu_0) + (1-\varphi) \mu_2 + \varphi (\xi_1+\lambda_1) &=& 0, \\
\label{t1d}
(1-\varphi) (\lambda_0+\xi_0) - (1-\varphi) \lambda_2 - \varphi (\eta_1+\mu_1) &=& 0, \\ 
\label{t1}
 - (1-\varphi) \xi_{k-1} + (1-\varphi) \lambda_{k+1} + \varphi (\eta_k-\mu_k) &=& 0,  \\
 \label{t2}
(1-\varphi) \eta_{k-1} + (1-\varphi) \mu_{k+1} - \varphi (\lambda_k-\xi_k)  &=& 0, \\
\label{t3}
 - (1-\varphi) \eta_{k-1} + (1-\varphi) \mu_{k+1} + \varphi (\xi_k + \lambda_k) &=& 0, \\
\label{t4}
 (1-\varphi) \xi_{k-1} - (1-\varphi) \lambda_{k+1} - \varphi (\eta_k+\mu_k) &=& 0 , 
\end{eqnarray}
for every $k\geq 2$. By adding and subtracting (\ref{t1}) and (\ref{t4}), and (\ref{t2}) and (\ref{t3}), we can rewrite the equations (\ref{t1}) to (\ref{t4}) as
\begin{eqnarray}
\label{u1}
 (1-\varphi) \xi_{k-1} - \varphi \mu_k &=& 0,  \\
\label{u3}
(1-\varphi) \lambda_{k+1} + \varphi \eta_k &=& 0, \\
\label{u2}
(1-\varphi) \mu_{k+1} + \varphi \xi_k &=& 0, \\
\label{u4}
 (1-\varphi) \eta_{k-1} - \varphi \lambda_k &=& 0 , 
\end{eqnarray} 
for every $k\geq 2$.

Let $k\geq 3$. From (\ref{u1}) and from (\ref{u2}) for $k-1$, we deduce that
\[
\varphi^2 \mu_k = \varphi (1-\varphi) \xi_{k-1} = - (1-\varphi)^2 \mu_k,
\]
which implies that
\[
\mu_k = 0, \qquad \forall k\geq 3.
\]
Together with (\ref{u1}) and (\ref{u2}), this implies that
\[
\xi_k = 0, \qquad \forall k\geq 2.
\]
Similarly, (\ref{u3}) for $k-1$ and (\ref{u4}) imply that
\[
\varphi^2 \lambda_k = \varphi (1-\varphi) \eta_{k-1} = - (1-\varphi)^2 \lambda_k,
\]
from which
\[
\lambda_k = 0, \qquad \forall k\geq 3,
\]
and, using again (\ref{u3}) and (\ref{u4}), 
\[
\eta_k = 0, \qquad \forall k\geq 2.
\]
If we add and subtract (\ref{t1a}) and (\ref{t1d}), and we do the same with (\ref{t1b}) and (\ref{t1c}), we can rewrite the equations (\ref{t1a}) to (\ref{t1d}) as
\begin{eqnarray}
\label{y1}
(1-\varphi) (\lambda_0 + \xi_0) - \varphi \mu_1 &=& 0, \\
\label{y2}
(1-\varphi) \lambda_2 + \varphi \eta_1 &=& 0, \\
\label{y3}
(1-\varphi) \mu_2 + \varphi \xi_1 &=& 0,\\
\label{y4}
(1-\varphi) (\eta_0 + \mu_0) - \varphi \lambda_1 &=& 0.
\end{eqnarray}
By (\ref{u1}) for $k=2$ and (\ref{y3}), we find
\[
\varphi^2 \mu_2 = \varphi(1-\varphi) \xi_1 = - (1-\varphi)^2 \mu_2,
\]
which implies that 
\[
\mu_2 = 0,
\]
and, using again (\ref{u1}) for $k=2$ and (\ref{y3}),
\[
\xi_1 = 0.
\]
Similarly, using (\ref{u4}) for $k=2$ and (\ref{y2}), we find
\[
\lambda_2 =0, \qquad \eta_1 = 0.
\]
This shows that the infinite system (\ref{t0a}) to (\ref{t4}) reduces to the system of four equations (\ref{t0a}), (\ref{t0b}), (\ref{y1}), (\ref{y4}) in the unknowns $\xi_0$, $\lambda_0$, $\eta_0$, $\mu_0$, $\lambda_1$, $\mu_1$, all the other variables being zero. The space of solutions of such a system is two-dimensional and consists of the 6-tuples $(\xi_0, \lambda_0, \eta_0, \mu_0, \lambda_1, \mu_1)$ whose first four components are related by the identities
\begin{eqnarray*} 
(1-\varphi)^2 (\eta_0 + \mu_0) + \varphi^2 (\eta_0 - \mu_0) &=& 0, \\
(1-\varphi)^2 (\lambda_0 + \xi_0) + \varphi^2 (\xi_0 - \lambda_0) &=& 0,
\end{eqnarray*}
and the last two are determined by the formulas
\begin{eqnarray*}
\lambda_1 &=& (1-2\varphi) \eta_0 + \mu_0, \\
\mu_1 &=& (1-2\varphi) \xi_0 + \lambda_0.
\end{eqnarray*}
Going back to the original variables, we conclude that the space of solutions of the infinite system (\ref{s}) is two-dimensional and consists of the infinite complex vectors $(x_k)_{k\in \N}$, $(y_k)_{k\in \N}$ such that $x_0$ and $y_0$ are related by the identities
\begin{eqnarray*}
(1-\varphi)^2 \re y_0 + \varphi^2 \im x_0 &=& 0,\\
(1-\varphi)^2 \re x_0 - \varphi^2 \im y_0 &=& 0,
\end{eqnarray*}
$x_1$ and $y_1$ are given by
\begin{eqnarray*}
x_1 &=& (1-\varphi) \re y_0 - \varphi \im x_0 - i \bigl( (1-\varphi) \re x_0 + \varphi \im y_0 \bigr), \\
y_1 &=& (1-\varphi) \re x_0 + \varphi \im y_0  + i \bigl( (1-\varphi) \re y_0 - \varphi \im x_0 \bigr),
\end{eqnarray*}
and $x_k=y_k=0$ for every $k\geq 2$. Therefore, a basis of this two-dimensional space of solutions is given by the vector
\[
\begin{array}{rclcrcll}
x_0 &=& \varphi^2 + (1-\varphi)^2 i, &\quad& y_0 &=& - \varphi^2 + (1-\varphi)^2 i,& \\ x_1 &=& \varphi(1-\varphi) (-1-i), &\quad& y_1 &=& \varphi(1-\varphi) (1-i) , &\\ x_k &=& 0, &\quad& y_k &=& 0, & \forall k\geq 2, \end{array}
\]
and by the vector
\[
\begin{array}{rclcrcll}
x_0 &=& \varphi^2 - (1-\varphi)^2 i, &\quad& y_0 &=&  \varphi^2 + (1-\varphi)^2 i,& \\ x_1 &=& \varphi(1-\varphi) (1-i), &\quad& y_1 &=& \varphi(1-\varphi) (1+i) , &\\ x_k &=& 0, &\quad& y_k &=& 0, & \forall k\geq 2, \end{array}
\]
We have proved that the real part of (\ref{vect}) vanishes if and only if the infinite vector $(u_k)_{k\in \N} = ((x_k,y_k))_{k\in \N}$ belongs to the plane generated by the above two vectors. Together with the identity (\ref{nucleo}), this implies that for every $r\in [0,1/2]$ the kernel of $D_{T_r}^+$ is two-dimensional and is generated 
by the pair of functions
\begin{eqnarray*}
u_r(s,t) & := & W_r(s,t) \Bigl( \hat{u}_0(r) e^{-\pi s} + \hat{u}_1 (r) e^{2\pi i t} e^{-3\pi s} \Bigr), \\
v_r(s,t) & := & W_r(s,t) \Bigl( \hat{v}_0(r) e^{-\pi s} + \hat{v}_1 (r) e^{2\pi i t} e^{-3\pi s} \Bigr),
\end{eqnarray*}
where
\[
\begin{array}{rlrl}
\hat{u}_0(r) :=& \varphi(r)^2 \binom{1}{-1} + (1-\varphi(r))^2 \binom{i}{i}, & 
\hat{v}_0(r) :=& \varphi(r)^2 \binom{1}{1} + (1-\varphi(r))^2 \binom{-i}{i} , \\
\hat{u}_1(r) :=& \varphi(r)(1-\varphi(r)) \binom{ -1-i }{ 1-i }, &
\hat{v}_1(r) :=& \varphi(r)(1-\varphi(r)) \binom{1-i }{ 1+i}.
\end{array}
\]
Since $D_{T_r}^+$ belongs to $\Sigma^+(-\pi I)$, it has Fredholm index two and, having a two-dimensional kernel, it is onto. This concludes the proof of Lemma \ref{kernel1}. 
\end{proof}

\begin{lem}
\label{kernel2}
If $1/2\leq r \leq 1$ then $D_{T_r}^+$ is onto and its kernel is the two-dimensional space which is generated  by the pair of functions
\[
u_r(s,t)  :=   W_r(s,t) \tilde{u}_0(r) e^{-\pi s}, \qquad
v_r(s,t)  :=  W_r(s,t) \tilde{v}_0(r) e^{-\pi s},
\]
where
\[
\tilde{u}_0(r) := \left( \begin{array}{c} e^{-i \frac{\pi}{2} \varphi(r-\frac{1}{2})} \\ - e^{i \frac{\pi}{2} \varphi(r-\frac{1}{2})} \end{array} \right), \qquad 
\tilde{v}_0(r) := \left( \begin{array}{c} e^{-i \frac{\pi}{2} \varphi(r-\frac{1}{2})} \\ e^{i \frac{\pi}{2} \varphi(r-\frac{1}{2})} \end{array} \right).
\]
\end{lem}

\begin{proof}
Set $u_k=(x_k,y_k)$ with $x_k,y_k\in \C$. Since $0\leq r \leq 1/2$, we have
\begin{equation}
\label{vect2}
W(r,t) \sum_{k\geq 0} e^{2\pi i k t} u_k = \sum_{k\geq 0} \left( \begin{array}{c} - i e^{i \frac{\pi}{2} \varphi(r-\frac{1}{2})} e^{2\pi i k t} x_k \\ i e^{-i \frac{\pi}{2} \varphi(r-\frac{1}{2})} e^{2\pi i k t} y_k \end{array} \right).
\end{equation}
The real part of the first component of (\ref{vect2}) equals
\begin{eqnarray*}
\sum_{k\geq 0} \Bigl( \bigl( \sin ( \textstyle{\frac{\pi}{2}} \varphi(r-\frac{1}{2})) \re x_k + \cos ( \frac{\pi}{2} \varphi(r-\frac{1}{2})) \im x_k\bigr) \cos (2\pi k t) \\ + \bigl( \cos (\textstyle{\frac{\pi}{2}}\varphi(r-\frac{1}{2})) \re x_k - \sin (\frac{\pi}{2} \varphi(r-\frac{1}{2})) \im x_k \bigr) \sin (2\pi k t) \Bigr).
\end{eqnarray*}
This expression vanishes for every $t\in \T$ if and only if 
\[
x_k = 0 , \qquad \forall k\geq 1,
\]
and $x_0$ satisfies
\begin{equation}
\label{z1}
\sin \bigl( \textstyle{\frac{\pi}{2}} \varphi(r-\frac{1}{2}) \bigr) \re x_0 + \cos \bigl( \frac{\pi}{2} \varphi(r-\frac{1}{2}) \bigr) \im x_0 = 0.
\end{equation}
The real part of the second component of (\ref{vect2}) equals
\begin{eqnarray*}
\sum_{k\geq 0} \Bigl( \bigl( \sin ( \textstyle{\frac{\pi}{2}} \varphi(r-\frac{1}{2})) \re y_k - \cos ( \frac{\pi}{2} \varphi(r-\frac{1}{2})) \im y_k\bigr) \cos (2\pi k t) \\ - \bigl( \cos ( \textstyle{\frac{\pi}{2}} \varphi(r-\frac{1}{2})) \re y_k + \sin (\frac{\pi}{2} \varphi(r-\frac{1}{2})) \im y_k \bigr) \sin (2\pi k t) \Bigr).
\end{eqnarray*}
This expression vanishes for every $t\in \T$ if and only if 
\[
y_k = 0 , \qquad \forall k\geq 1,
\]
and $y_0$ satisfies
\begin{equation}
\label{z2}
\sin \bigl( \textstyle{\frac{\pi}{2}} \varphi(r-\frac{1}{2}) \bigr) \re y_0 - \cos \bigl( \frac{\pi}{2} \varphi(r-\frac{1}{2}) \bigr) \im y_0 = 0.
\end{equation}
Therefore, the real part of (\ref{vect2}) vanishes if and only if $u_k=0$ for every $k\geq 1$ and $u_0\in \C^2$ belongs to the two-dimensional real subspace which is spanned by the vectors $\tilde{u}_0(r)$ and $\tilde{v}_0(r)$. The conclusion follows from the identity (\ref{nucleo}) and from the fact that $D_{T_r}^+$ has Fredholm index two.
\end{proof}

The above two lemmas allow us to understand the effect of the conjugacy by the unitary map $W$ on the determinant bundle of the space $\Sigma^+(-\pi I)$:

\begin{lem}
\label{particolare}
Let $n=2$ and let $W\in C^{\infty}([0,+\infty]\times \T,\mathrm{U}(n))$ be the map which is defined above. Then the canonical lift to the determinant bundle of the map
\[
\Sigma^+(-\pi I) \rightarrow \Sigma^+(-\pi I), \qquad D_A^+ \mapsto W D_A^+ W^{-1},  
\]
is orientation-reversing.
\end{lem}

\begin{proof}
Since
\[
\hat{u}_0\bigl(\textstyle\frac{1}{2}\bigr) = \displaystyle\binom{1}{-1}  = \tilde{u}_0\bigl(\textstyle{\frac{1}{2}}\bigr),\qquad \hat{v}_0\bigl(\textstyle{\frac{1}{2}}\bigr) = \displaystyle\binom{1}{1} = \tilde{v}_0\bigl(\textstyle{\frac{1}{2}}\bigr),
\]
and
\[
\hat{u}_1\bigl(\textstyle{\frac{1}{2}}\bigr) = \hat{v}_1\bigl(\textstyle{\frac{1}{2}}\bigr) = 0,
\]
the basis $u_r$, $v_r$ of $\ker D_{T_r}^+$ which is defined in Lemmas \ref{kernel1} and \ref{kernel2} depends continuously on $r\in [0,1]$. Notice also that 
\begin{equation}
\label{rel}
Wu_1 = -u_0, \qquad W v_1 = v_0.
\end{equation} 
Let us fix one of the two orientations of $\det(\Sigma^+(-\pi I))$, for instance the one for which $u_1 \wedge v_1$ is a positive generator of the line
\[
\det \bigl(D_{T_1}^+\bigr) = \det \bigl(D_{-\pi I}^+\bigr) = \Lambda^2 \bigl( \ker  D_{-\pi I}^+ \bigr).
\]
Since $T_r$ is onto for every $r\in [0,1]$ and its kernel is generated by the pair of vectors $u_r$, $v_r$, which depend continuously on $r$, $u_0 \wedge v_0$ is a positive generator of the line
\[
\det \bigl(D_{T_0}^+\bigr) = \det \bigl(W D_{-\pi I}^+ W^{-1}\bigr) = \Lambda^2 \bigl( \ker W D_{-\pi I}^+ W^{-1} \bigr).
\]
On the other hand, the generator of the above line
which is canonically induced by $u_1 \wedge v_1$ by means of conjugacy of $D_{-\pi I}^+$ by $W$ is 
\[
(Wu_1) \wedge (Wv_1) = (-u_0) \wedge v_0 = - u_0 \wedge v_0,
\]
where we have used (\ref{rel}). Since this generator is negative, the canonical map
\[
\det \bigl(D_{-\pi I}^+\bigr) \rightarrow \det \bigl(W D_{-\pi I}^+ W^{-1}\bigr)
\]
is orientation reversing. Then same fact is true for every $D_A^+\in \Sigma^+(-\pi I)$ because $\Sigma^+(-\pi I)$ is connected.
\end{proof}

We are finally ready to prove Proposition \ref{complex}:

\begin{proof}[Proof of Proposition \ref{complex}]
When $n=1$, $\mathrm{SO}(1)=\{I\}$, so the assumption of Lemma \ref{contr} is trivially satisfied and case (i) of Proposition \ref{complex} holds. 

Let $n=2$ and let $W$ be as above. By Lemma \ref{contr} and Lemma \ref{particolare}, together with the functoriality of the canonical mapping between determinant bundles which is induced by composition by isomorphisms, we deduce that the canonical lift to the determinant bundle of the map
\[
\Sigma^+(-\pi I_{\C^2}) \rightarrow \Sigma^+(-\pi I_{\C^2}), \qquad D_{A}^+ \mapsto W^k D_A^+ W^{-k}, \qquad k\in \Z,
\]
is orientation-preserving if and only if $k$ is even. 

Since for $n=2$ an arbitrary $U$ as in the hypothesis is of the form $U=V\cdot W^k$, where $V$ satisfies the hypotheses of Lemma \ref{contr} and $k\in \Z$, we deduce that case (ii) of Proposition \ref{complex} holds in the particular case $S^+=-\pi I_{\C^2}$.

Now let $n>2$. If $U=W \oplus I_{\C^{n-2}}$ and $S^+=-\pi I_{\C^n}$,  the canonical lift to the determinant bundle of the map
\[
\Sigma^+(-\pi I_{\C^n}) \rightarrow \Sigma^+(-\pi I_{\C^n}), \qquad D_{A}^+ \mapsto U D_A^+ U^{-1}, 
\]
is orientation reversing, as one deduces from Lemma \ref{particolare} and from the fact that within $\Sigma^+(-\pi I_{\C^n})$ one has maps which preserve the splitting $\C^2 \oplus \C^{n-2}$.
Since an arbitrary $U$ as in the hypothesis is either such that $U(0,\cdot)$ is contractible within $\mathrm{SO}(n)$, or  of the form $V \cdot (W\oplus I_{\C^{n-2}})$ where $V$ satisfies the hypotheses of Lemma \ref{contr}, we deduce that case (iii) of Proposition \ref{complex} holds in the particular case $S^+=-\pi I_{\C^n}$.

Therefore, Proposition \ref{complex} holds for every dimension when $S^+=-\pi I$. The case of a general $S^+$ follows by gluing.
\end{proof}


\providecommand{\bysame}{\leavevmode\hbox to3em{\hrulefill}\thinspace}
\providecommand{\MR}{\relax\ifhmode\unskip\space\fi MR }
\providecommand{\MRhref}[2]{%
  \href{http://www.ams.org/mathscinet-getitem?mr=#1}{#2}
}
\providecommand{\href}[2]{#2}

\end{document}